\documentclass[11pt]{article}

\usepackage[margin=1in]{geometry}
\usepackage[T1]{fontenc}
\usepackage[utf8]{inputenc}
\usepackage{lmodern}
\usepackage{microtype}
\usepackage{amsmath,amssymb,amsthm,mathtools}
\usepackage{enumitem}
\usepackage{titlesec}
\usepackage{comment}
\usepackage[colorlinks=true,linkcolor=blue,citecolor=blue,urlcolor=blue]{hyperref}

\setlist{nosep,leftmargin=*}
\allowdisplaybreaks
\titlespacing*{\section}{0pt}{2ex plus .5ex minus .2ex}{1ex plus .2ex}
\titlespacing*{\subsection}{0pt}{1.4ex plus .4ex minus .2ex}{.7ex plus .2ex}

\theoremstyle{plain}
\newtheorem{theorem}{Theorem}[section]
\newtheorem{lemma}[theorem]{Lemma}

\newtheorem{problem}[theorem]{Problem}
\theoremstyle{definition}

\newtheorem{remark}[theorem]{Remark}

\newcommand{\cE}{\mathcal{E}}
\newcommand{\one}{\mathbf{1}}

\title{Covering the ternary cube by binary subcubes}
\author{Peiru Kuang\footnote{School of Mathematical Sciences, Shanghai Jiao Tong University, Shanghai 200240, China. Supported by SJTU-Warwick Joint Seed Fund. Email: peiru\_k@sjtu.edu.cn}
\and
Yan Wang\footnote{School of Mathematical Sciences, Shanghai Jiao Tong University, Shanghai 200240, China. Supported by National Key R\&D Program of China under Grant No. 2022YFA1006400, National Natural Science Foundation of China under Grant No. 12571376 and SJTU-Warwick Joint Seed Fund. Email: yan.w@sjtu.edu.cn (corresponding author).}}
\date{}

\begin{document}
\maketitle
\begin{abstract}
For an integer $n\ge0$, let $f(n)$ be the minimum number of subcubes of $\mathbb{Z}_3^n$ of the form $A_1\times\cdots\times A_n$, where $|A_i|=2$ for every $i$, whose union covers $\mathbb{Z}_3^n$. A simple counting argument gives $f(n)\ge(3/2)^n$, while  $f(n)=O(n(3/2)^n)$ by random construction. 
We prove that $f(n)\le2(3/2)^n-1$, answering a problem of Imre Leader. 
We also show that $f(n)/(3/2)^n$ is nondecreasing and there exists a constant
$C_3$ such that $f(n)=(C_3+o(1))(3/2)^n$ where
$1.62227<C_3\le2$.
\end{abstract}

\section{Introduction}

For an integer $n$, let $[n]=\{1,\cdots, n\}$. A \emph{binary subcube} of
$\mathbb{Z}_3^n$ is a set of the form
$A_1\times\cdots\times A_n$, where $A_i\in\binom{\mathbb{Z}_3}2$ for every $i\in[n]$.
Let $f(n)$ denote the minimum size of a cover of $\mathbb{Z}_3^n$ by binary subcubes. The following problem was posed by Imre Leader, and communicated to us by Oleg Pikhurko.

\begin{problem}[Leader]
     Determine the order of magnitude of $f(n)$.
\end{problem}

For $a\in\mathbb{Z}_3^n$, let
$Q_a=\{x\in\mathbb{Z}_3^n:x_i\ne a_i\text{ for every }i\in[n]\}$. Since $|Q_a|=2^n$, a simple counting argument gives
\begin{equation}\label{eq:intro-lower}
 f(n)\ge \left\lceil\frac{3^n}{2^n}\right\rceil
 \ge \left(\frac32\right)^n.
\end{equation}

A random construction or a greedy argument gives an upper bound that $f(n)\le (\ln 3+o(1))n\left(\frac32\right)^n$.
A slightly stronger bound follows by formulating the problem as a hypergraph covering problem (also as coverings of groups by translates, see Bollob\'as, Janson and Riordan~\cite{BollobasJansonRiordan2011}).
For a hypergraph $\mathcal G$, a \emph{cover} is a family
$\mathcal C\subseteq E(\mathcal G)$ such that
$V(\mathcal G)=\bigcup_{C\in\mathcal C}C$,
and the covering number $k(\mathcal G)$ is the minimum size of such a family. A \emph{fractional cover} of $\mathcal G$ is a function
$w:E(\mathcal G)\to\mathbb R_{\ge0}$ such that
$\sum_{\substack{C\in E(\mathcal G), v\in C}}w(C)\ge1$ for every $v\in V(\mathcal G)$.
The fractional covering number is defined by
$
k_f(\mathcal G)
=
\min\left\{
\sum_{C\in E(\mathcal G)}w(C):
w\text{ is a fractional cover of }\mathcal G
\right\}.
$
The relation between integral and fractional covers was studied in the classical work of Johnson~\cite{Johnson1974} and Lov\'asz~\cite{Lovasz1975}; see also Berge~\cite{Berge1978} and
Scheinerman and Ullman~\cite{ScheinermanUllman}.
Let $\mathcal H_n$ be the hypergraph with vertex set $\mathbb{Z}_3^n$ and edge set $\{Q_a:a\in\mathbb{Z}_3^n\}$. 
The covering estimates of Johnson~\cite{Johnson1974},
Lov\'asz~\cite{Lovasz1975} and Stein~\cite{Stein1974} now give
\begin{equation}\label{eq:intro-jls}
f(n)=k(\mathcal H_n)
\le
(1+\ln 2^n)k_f(\mathcal H_n)
=
(1+n\ln 2)\left(\frac32\right)^n.
\end{equation}
Thus these estimates leave a gap of a factor of order $n$ between
\eqref{eq:intro-lower} and \eqref{eq:intro-jls}.

In this paper, we determine $f(n)$ up to a constant factor and show that the limit of  $\frac{f(n)}{(3/2)^n}$ exists.

\begin{theorem}\label{thm:main}
For every integer $n\ge0$,
$$
f(n)\le2\left(\frac32\right)^n-1.
$$
Moreover, the sequence
$f(n)/(3/2)^n$ is nondecreasing. Thus, the limit
$C_3=\lim_{n\to\infty}\frac{f(n)}{(3/2)^n}$
exists and $1.62227<C_3\le2$.
\end{theorem}
It remains open to determine the exact value of $C_3$. We do not try to optimize it in this paper.

Theorem~\ref{thm:main} has a natural application to the total domination in direct products of graphs. 
A set \(D\subseteq V(G)\) is a \emph{total dominating set} of a graph \(G\) if \(N_G(v)\cap D\ne\varnothing\) for every \(v\in V(G)\).  
The minimum cardinality of such a set is the \emph{total domination number} \(\gamma_t(G)\).  
The graph \(K_q^{\times n}\) has vertex set \([q]^n\), with two vertices adjacent if and only if they differ in every coordinate.
Thus, after identifying $[3]$ with $\mathbb Z_3$, the open neighborhood of $a\in\mathbb Z_3^n$ is precisely $Q_a$, and hence $f(n)=\gamma_t(K_3^{\times n})$.
Theorem~\ref{thm:main} therefore determines the total domination number of \(K_3^{\times n}\)
within a factor of two. Total domination in direct products was initiated by Rall~\cite{Rall2005} and was subsequently studied by Dorbec, Gravier, Klav\v{z}ar and
\v{S}pacapan~\cite{DorbecGravierKlavzarSpacapan2006}. 
Direct products of complete graphs were considered further by Meki\v{s}~\cite{Mekis2010}. 
Vemuri~\cite{Vemuri2020} observed that the available bounds are far from tight when many of the factors are small and their number is large.
The powers \(K_3^{\times n}\) provide a basic example of this regime. Since \(K_2^{\times n}\) is a perfect matching, \(K_3^{\times n}\) is the first nontrivial case.

More recently, Adriaensen, Ihringer, Martin and
Villagr\'an~\cite{AdriaensenIhringerMartinVillagran} studied \(K_q^{\times n}\) using skirting sets, motivated by a Hamming-ball lemma of Alon, Ben-Eliezer, Shangguan and Tamo~\cite{AlonBenEliezerShangguanTamo}. 
For \(x,y\in[q]^n\), let
\(d_H(x,y)=|\{i\in[n]:x_i\ne y_i\}|\) denote their Hamming distance. A set \(S\subseteq[q]^n\) is a \textit{skirting set} if, for every \(y\in[q]^n\), there exists \(x\in S\) such that \(d_H(x,y)=n\). 
Let \(f(n,q)\) denote the minimum cardinality of a skirting set in \([q]^n\). The definition immediately gives $f(n,q)=\gamma_t(K_q^{\times n})$,
and in particular \(f(n)=f(n,3)\). Combining their equality \(f(6,3)=18\) with our upper bound yields $41\le f(8,3)\le 50$, improving their bounds \(29\le f(8,3)\le 54\).

A closely related but different problem concerns
partitions of discrete boxes. If $A=A_1\times\cdots\times A_d$ is a finite discrete box, a sub-box $B=B_1\times\cdots\times B_d$ is called \emph{proper} if $B_i\ne A_i$ for every $i\in[d]$. Alon, Bohman, Holzman and Kleitman \cite{AlonBohmanHolzmanKleitman} proved that every partition of a $d$-dimensional discrete box into proper sub-boxes has at least $2^d$ parts. Further variants were studied by Buci\'c, Lidick\'y, Long and Wagner~\cite{BucicLidickyLongWagner}. The disjointness condition is essential here. Indeed, for $n\ge2$, Theorem~\ref{thm:main} gives a cover of $\mathbb{Z}_3^n$ by fewer than $2^n$ binary subcubes, whereas every partition of $\mathbb{Z}_3^n$ into proper sub-boxes has at least $2^n$ parts. Thus allowing overlaps (from partition to cover) changes the required number of subsets from $2^n$ to $(3/2)^n$.

We conclude the introduction with a brief outline of the proof. Let $G$ be a graph on the set of coordinates, and let $\cE(G)$ be the family of sets $S\subseteq V(G)$ for which every vertex of $G[S]$ has even degree. We associate a binary subcube with every set in $\cE(G)$ and prove that these subcubes cover $\mathbb Z_3^{V(G)}$. 
In Section 2, we show that $f(n)\le |\cE(G)|$
for every graph $G$ on $n$ vertices. Then we take $G\sim G(n,1/2)$. For every fixed nonempty set
$S\subseteq[n]$ of size $s$, we have $\Pr\bigl(S\in\cE(G)\bigr)=2^{-(s-1)}$.
Therefore, we obtain
$
 \mathbb E|\cE(G)|
 =1+\sum_{s=1}^n\binom ns2^{-(s-1)}
 =2\left(\frac32\right)^n-1,
$
which gives the upper bound in Theorem~\ref{thm:main} by averaging.
For the lower bound, we first prove the projection inequality that $f(n)\ge \left(\frac32\right)^{n-k}f(k)$ for $0\le k\le n$.
This shows that $f(n)/(3/2)^n$ is nondecreasing. Finally, retaining the integer rounding in the projection inequality and iterating give $C_3>1.622270502884767$.

\section{Even induced subgraphs}
The upper bound will be obtained from an auxiliary graph. Let \(G\) be a graph whose vertex set indexes the coordinates of the ternary cube.
To each \(S\subseteq V(G)\) for which \(G[S]\) has only even degrees, we associate a vector \(a^S\in\mathbb{Z}_3^{V(G)}\). The key observation is that the subcubes \(Q_{a^S}\) obtained in this way cover \(\mathbb{Z}_3^{V(G)}\). 
Therefore, it remains to choose \(G\) with few induced subgraphs having only even degrees.

Let $G=(V,E)$ be a finite simple graph. For $S\subseteq V$ and $v\in V$,
let
\(
d_S(v)=|N_G(v)\cap S|.
\)
We call $S$ an \emph{even vertex set} if every vertex of $G[S]$ has even degree; equivalently, $d_S(v)\equiv0\pmod 2$ for every $v\in S$. Let
$
\cE(G)=\{S\subseteq V: S\text{ is an even vertex set}\}$.
We first need the following parity lemma to select an appropriate even vertex set.
\begin{lemma}\label{lem:parity}
Let $H$ be a finite graph, and let $B\subseteq V(H)$. There exists
$S\subseteq V(H)$ such that 
$d_S(v)\equiv0\pmod2$ for every $v\in V(H)\setminus B$ and $d_S(v)+\one_S(v)\equiv1\pmod2$ for every $v\in B$, where $\one_S$ denotes the indicator function of $S$.
\end{lemma}

\begin{proof}
For each \(u\in V(H)\), write
\(s_u=\one_S(u)\in\mathbb Z_2\).
We seek a set \(S\subseteq V(H)\) such that the following holds. 
$$
\sum_{u\in N_H(v)}s_u \equiv 0 \pmod2 \quad (v\notin B),
\qquad
s_v+\sum_{u\in N_H(v)}s_u\equiv 1 \pmod2 \quad (v\in B).
$$
Such a solution gives the desired set
\(S=\{u\in V(H):s_u=1\}\).

Suppose that the system has no solution. Then there exists a collection of its
equations that adds to the contradiction \(0\equiv1 \pmod2 \). Let
\(T\subseteq V(H)\) be the set of vertices that indexes these equations.

For \(w\in T\), the variable \(s_w\) occurs once for each neighbor of \(w\) in \(T\), and once more if \(w\in B\). Since its total coefficient is zero on the left-hand side, \(|N_H(w)\cap T|\) is odd if and only if \(w\in B\). Thus, the odd-degree vertices of \(H[T]\) are precisely the vertices of \(T\cap B\).
On the other hand, the sum of the right-hand sides is
\(|T\cap B|\) modulo \(2\), and hence \(|T\cap B|\) is odd. This
contradicts the fact that every graph has an even number of
odd-degree vertices. Therefore, the system has a solution.
\end{proof}

For each $S\in\cE(G)$, define $a^S\in\mathbb{Z}_3^V$ by
\begin{equation}\label{eq:def-aS}
a^S_v=
\begin{cases}
0,&v\in S,\\
1,&v\notin S\text{ and }d_S(v)\text{ is odd},\\
2,&v\notin S\text{ and }d_S(v)\text{ is even}.
\end{cases}
\end{equation}
Define
$$
\mathcal A_G
=
\{a^S:S\in\mathcal E(G)\}
\subseteq\mathbb{Z}_3^V.
$$
We next show that the vectors \(a^S\) arising from even vertex sets give a cover of the ternary cube.
\begin{lemma}\label{lem:graph-cover}
For every finite graph $G=(V,E)$, the family
$
\{Q_{a^S}:S\in\mathcal E(G)\}
$
covers $\mathbb{Z}_3^V$. Thus, $|\mathcal A_G|=|\cE(G)|$ and $f(n)\le |\cE(G)|$ if $|V|=n$.
\end{lemma}

\begin{proof}
The map $S\mapsto a^S$ is injective, since
$
S=\{v\in V:a^S_v=0\}.
$
Thus $|\mathcal A_G|=|\cE(G)|$.
It remains to prove the covering property. Fix $x\in\mathbb{Z}_3^V$, and let $V_i=\{v\in V:x_v=i\}$ for $i\in\{0,1,2\}$.
Let $Y=V_1\cup V_2$ and $H=G[Y]$. Apply Lemma~\ref{lem:parity} to $H$ with $B=V_2$. We obtain $S\subseteq Y$ such that
\begin{align}
d_S(v)&\equiv0\pmod2 &&(v\in V_1),\label{eq:cover-v1}\\
d_S(v)+\one_S(v)&\equiv1\pmod2 &&(v\in V_2).\label{eq:cover-v2}
\end{align}
Since $S\subseteq Y$, the quantities $d_S(v)$ are the same whether they are computed in $H$ or in $G$.

We first verify that $S\in\cE(G)$. If $v\in S\cap V_1$, then
\eqref{eq:cover-v1} gives $d_S(v)\equiv0\pmod2$. If $v\in S\cap V_2$,
then $\one_S(v)=1$, so \eqref{eq:cover-v2} gives
$d_S(v)\equiv0\pmod2$. Thus every vertex of $G[S]$ has even degree.

Now we show that $a^S\neq x$ coordinate-wise.
If $v\in V_0$, then $v\notin S$, so $a^S_v\in\{1,2\}$ and
$a^S_v\ne x_v$.
If $v\in V_1$, then either $v\in S$ and $a^S_v=0$, or $v\notin S$
and \eqref{eq:cover-v1} gives $a^S_v=2$. Hence $a^S_v\ne x_v$.
If $v\in V_2$, then either $v\in S$ and $a^S_v=0$, or $v\notin S$
and \eqref{eq:cover-v2} gives that $d_S(v)$ is odd, so $a^S_v=1$. Hence
$a^S_v\ne x_v$.

Therefore $a^S_v\ne x_v$ for every $v\in V$, so $x\in Q_{a^S}$. Since $x$ can be chosen arbitrarily, the family
$
\{Q_{a^S}:S\in\mathcal E(G)\}
$
covers $\mathbb Z_3^V$.
\end{proof}

\section{Upper bound}
\begin{proof}[Proof of the upper bound in Theorem~\ref{thm:main}]
We now choose \(G\sim G(n,1/2)\) on the vertex set \([n]\).
For a nonempty set \(S\subseteq[n]\), let \(s=|S|\). We claim that
$$
\Pr\bigl(S\in\mathcal E(G)\bigr)=2^{-(s-1)}.
$$
Indeed, fix a spanning tree \(T\) of the complete graph on \(S\), and
condition on all edges of \(G[S]\) outside \(T\). Choose an arbitrary vertex as root. Starting from the leaves and proceeding towards the
root, we consider each non-root vertex \(v\). At the moment when \(v\) is
processed, all edges incident with \(v\), except possibly the edge
joining \(v\) to its parent, have already been determined. Hence that
parent edge has a unique value which makes the degree of \(v\) even.

This uniquely determines all \(s-1\) tree edges. The degree of the root
is then also even, since every graph has an even number of vertices of
odd degree. Thus, for every fixed choice of the non-tree edges, exactly
one of the \(2^{s-1}\) choices of the tree edges makes all degrees in
\(G[S]\) even. Therefore
$
\Pr\bigl(S\in\mathcal E(G)\bigr)=2^{-(s-1)}.
$

Hence, by linearity of expectation, we have
\[
\begin{aligned}
\mathbb E |\cE(G)|
&=\sum_{S\subseteq[n]}\Pr\bigl(S\in\mathcal E(G)\bigr)
=1+\sum_{s=1}^n\binom ns2^{-(s-1)}
=1+2\sum_{s=1}^n\binom ns2^{-s}\\
&=1+2\left(\left(1+\frac12\right)^n-1\right)
=2\left(\frac32\right)^n-1.
\end{aligned}
\]
Thus, there exists a graph \(G\) on \([n]\) such that
$
|\cE(G)|\le 2\left(\frac32\right)^n-1.
$
By Lemma~\ref{lem:graph-cover},
$$
f(n)\le |\cE(G)|\le 2\left(\frac32\right)^n-1,
$$
as required.
\end{proof}

\section{Lower bound}

We begin with a projection inequality, which yields the lower bound and shows that the sequence \(f(n)/(3/2)^n\) is nondecreasing.

\begin{lemma}\label{lem:projection}
For all integers $0\le k\le n$,
$$
f(n)\ge \left(\frac32\right)^{n-k}f(k).
$$
\end{lemma}

\begin{proof}
Let $\mathcal A\subseteq\mathbb Z_3^n$ be such that
$\{Q_a:a\in\mathcal A\}$ covers $\mathbb{Z}_3^n$. 
Fix a set $I\subseteq[n]$ of size $k$, and let $J=[n]\setminus I$. 
For each
$y\in\mathbb Z_3^J$, let
$
N_y=
\left|
\left\{a\in\mathcal A:a_j\ne y_j\text{ for every }j\in J\right\}
\right|.
$
The projections onto $I$ of the vectors counted by $N_y$ form a cover of $\mathbb Z_3^I$. Indeed, for each $x\in\mathbb Z_3^I$, the point whose restrictions to $I$ and $J$ are $x$ and $y$, respectively, is covered by some $Q_a$. This vector $a$ is counted by $N_y$, and its projection onto $I$ differs from $x$ in every coordinate. It follows that $N_y\ge f(k)$ for every $y\in\mathbb{Z}_3^J$.

On the other hand, each $a\in\mathcal A$ is counted by exactly
$2^{n-k}$ of the numbers $N_y$. Therefore
$$
3^{n-k}f(k)
\le \sum_{y\in\mathbb{Z}_3^J}N_y
=2^{n-k}|\mathcal A|.
$$
Taking $|\mathcal A|=f(n)$ proves the lemma. 
\end{proof}




The projection inequality shows that the sequence $f(n)/(3/2)^n$ is nondecreasing. Together with the upper bound, this implies that the limit $C_3=\lim_{n\to\infty}f(n)/(3/2)^n$ exists and satisfies $C_3\le2$.

We next improve the lower bound by retaining the integer rounding in the projection inequality. Let $L_1=2$ and $L_{n+1}=\lceil3L_n/2\rceil$ for $n\ge1$. Since $f(1)=2$, Lemma~\ref{lem:projection} and induction give $f(n)\ge L_n$ for every $n\ge1$. Iterating the recurrence gives
$L_{100}=659552205933476787$, and hence
$C_3\ge L_{100}/(3/2)^{100}
=1.6222705028847673143\ldots>1.62227$.
This completes the proof of Theorem~\ref{thm:main}.

\begin{remark}
The integer iteration above has a limiting value. Indeed, recall that $L_{n+1}=\lceil 3L_n/2\rceil$ and
$a_n=L_n/(3/2)^n$. If $L_n$ is even, then
$L_{n+1}=3L_n/2$, and hence $a_{n+1}=a_n$. If $L_n$ is odd, then $L_{n+1}=(3L_n+1)/2=3L_n/2+1/2$, and therefore $a_{n+1}-a_n=\frac12(2/3)^{n+1}$. Then $a_{n+1}-a_n$ is either $0$ or $\frac12(2/3)^{n+1}$, according as $L_n$ is even or odd. Hence $(a_n)$ is nondecreasing and converges to a limit $\lambda$. Moreover, for every $N\ge1$, we have $0\le\lambda-a_N\le(2/3)^N$. Computing the recurrence up to $N=1000$ and using this error bound gives
$\lambda=1.622270502884767315956950982899324\ldots$.
Thus the integer iteration alone yields $C_3\ge\lambda$.
\end{remark}


\section{Concluding remarks}

Our proof uses a parity construction that is specific to the ternary setting. It would be interesting to understand the corresponding problem for larger alphabets.

For integers $q\geq2$ and $n\geq1$, and for $a\in\mathbb{Z}_q^n$, define
$
Q_a^{(q)}
=\{x\in\mathbb{Z}_q^n:x_i\ne a_i\text{ for every }i\in[n]\}
=\prod_{i=1}^n\bigl(\mathbb{Z}_q\setminus\{a_i\}\bigr)
$
and
$
f_q(n)
=\min\left\{
|\mathcal A|:
\mathcal A\subseteq\mathbb{Z}_q^n
\text{ and }
\mathbb{Z}_q^n=\bigcup_{a\in\mathcal A}Q_a^{(q)}
\right\}.
$
Thus $f_3(n)=f(n)$ while the case $q=2$ is trivial.

Each $Q_a^{(q)}$ has size $(q-1)^n$. Moreover, the associated covering
hypergraph is $(q-1)^n$-uniform and $(q-1)^n$-regular, so its fractional
covering number is exactly
$
\left(\frac{q}{q-1}\right)^n.
$
Thus,
$
f_q(n)\geq
\left\lceil
\left(\frac{q}{q-1}\right)^n
\right\rceil,
$
whereas the Johnson--Lov\'asz--Stein estimate gives
$
f_q(n)
\leq
\bigl(1+n\ln(q-1)\bigr)
\left(\frac{q}{q-1}\right)^n.
$
It is therefore natural to define
$
C_q
=
\sup_{n\geq1}
\frac{f_q(n)}
{\left(q/(q-1)\right)^n}.
$
Theorem~\ref{thm:main} shows that $C_3\leq2$.

\begin{problem}
Is $C_q$ finite for every fixed $q\geq3$? Moreover, can $C_q$ be bounded by an absolute constant independent of $q$?
\end{problem}

\section*{Acknowledgements}
We thank Oleg Pikhurko for bringing the problem to our attention. We also thank Jun Gao for a helpful discussion.

\section*{Declaration on the use of AI} 
Generative AI tools were used to simplify some proofs and check arguments. All mathematical results and proofs were independently reviewed and verified by the authors, who take full responsibility for the content of this work.


\begin{thebibliography}{99}

\bibitem{AdriaensenIhringerMartinVillagran}
S. Adriaensen, F. Ihringer, W. J. Martin, and R. R. Villagr\'an,
\emph{Skirting the \(n\)-tuples},
arXiv:2602.01080, 2026.

\bibitem{AlonBohmanHolzmanKleitman}
N. Alon, T. Bohman, R. Holzman, and D. J. Kleitman,
\emph{On partitions of discrete boxes},
Discrete Math. \textbf{257} (2002), no.~2--3, 255--258.

\bibitem{AlonBenEliezerShangguanTamo}
N. Alon, O. Ben-Eliezer, C. Shangguan, and I. Tamo,
\emph{The hat guessing number of graphs},
J. Combin. Theory Ser. B \textbf{144} (2020), 119--149.

\bibitem{Berge1978}
C. Berge,
\emph{Fractional Graph Theory},
ISI Lecture Notes \textbf{1},
Macmillan of India, 1978.

\bibitem{BollobasJansonRiordan2011}
B. Bollob\'as, S. Janson, and O. Riordan,
\emph{On covering by translates of a set},
Random Structures Algorithms \textbf{38} (2011), no.~1--2, 33--67.

\bibitem{BucicLidickyLongWagner}
M. Buci\'c, B. Lidick\'y, J. Long, and A. Z. Wagner,
\emph{Partition problems in high dimensional boxes},
J. Combin. Theory Ser. A \textbf{166} (2019), 315--336.

\bibitem{DorbecGravierKlavzarSpacapan2006}
P. Dorbec, S. Gravier, S. Klav\v{z}ar, and S. \v{S}pacapan,
\emph{Some results on total domination in direct products of graphs},
Discuss. Math. Graph Theory \textbf{26} (2006), no.~1, 103--112.


\bibitem{Johnson1974}
D. S. Johnson,
\emph{Approximation algorithms for combinatorial problems},
J. Comput. System Sci. \textbf{9} (1974), 256--278.


\bibitem{Lovasz1975}
L. Lov\'asz,
\emph{On the ratio of optimal integral and fractional covers},
Discrete Math. \textbf{13} (1975), no.~4, 383--390.

\bibitem{Mekis2010}
G. Meki\v{s},
\emph{Lower bounds for the domination number and the total domination
number of direct product graphs},
Discrete Math. \textbf{310} (2010), 3310--3317.

\bibitem{Rall2005}
D. F. Rall,
\emph{Total domination in categorical products of graphs},
Discuss. Math. Graph Theory \textbf{25} (2005), no.~1--2, 35--44.

\bibitem{ScheinermanUllman}
E. R. Scheinerman and D. H. Ullman,
\emph{Fractional Graph Theory: A Rational Approach to the Theory of
Graphs},
Wiley-Interscience Series in Discrete Mathematics and Optimization,
John Wiley \& Sons, New York, 1997.

\bibitem{Stein1974}
S. K. Stein,
\emph{Two combinatorial covering theorems},
J. Combin. Theory Ser. A \textbf{16} (1974), 391--397.

\bibitem{Vemuri2020}
H. Vemuri,
\emph{Domination in direct products of complete graphs},
Discrete Appl. Math. \textbf{285} (2020), 473--482.



\end{thebibliography}
\end{document}